\documentclass[a4paper]{article}
  \usepackage{amsfonts,amsmath,amssymb,amsthm}
  \usepackage{soul}
 \usepackage[colorlinks=true]{hyperref}
   \def\R{\mathbb{R}}
   \def\N{\mathbb{N}}
   
   \def\1{{\rm I\mskip -10.5mu 1}}

   \def\D{{\nabla}}

   \def\cC{{\cal C}}
   
   \def\cE{{\cal E}}

   \def\cM{{\cal M}}

   \def\meas{\mathop{\rm meas}\nolimits}
   \def\dist{\mathop{\rm dist}\nolimits}

   \def\loc{\mathop{\rm loc}\nolimits}

   \newcommand{\beq}{\begin{equation}}
   \newcommand{\eeq}{\end{equation}}

\newcommand{\la}{\lambda}

\newcommand{\wt}{\widetilde}

\newcommand{\rn}{\mathbb R^N}
\newcommand{\ir}{\int_{\mathbb R^N}}
\newcommand{\n}{\nabla}
\theoremstyle{definition}
\newtheorem{df}{Definition}[section]
\theoremstyle{remark}
\newtheorem{rem}[df]{Remark}
\theoremstyle{plain}
\newtheorem{prop}[df]{Proposition}
\newtheorem{lemma}[df]{Lemma}
\newtheorem{teo}[df]{Theorem}

 \newcommand{\sezione}[1]{\section{#1}\setcounter{equation}{0}}

\newcommand{\eps}{\varepsilon}

\begin{document}

%%%%%%%%%%%%%%%%%%%%%%%%        TITOLO       %%%%%%%%%%%%%%%%%%%%%%%%%%%%%

\title{Normalized solutions to mass supercritical Schr\"odinger equations with negative potential}

\author{
Riccardo Molle
\\
{\small\it Dipartimento di Matematica, Universit\`a di Roma ``Tor Vergata''}
\\
{\small\it Via della Ricerca Scientifica n. 1, 00133 Roma, Italy}
\\ \\
Giuseppe Riey
\\
{\small\it Dipartimento di Matematica e Informatica, Universit\`a della Calabria}
\\
{\small\it Via P. Bucci 31B, 87036 Rende (CS), Italy}
\\ \\
Gianmaria Verzini
\\
{\small\it Dipartimento di Matematica, Politecnico di Milano}
\\
{\small\it Piazza Leonardo da Vinci, 32, 20133 Milano, Italy}
}

\date{}
\maketitle

%%%%%%%%%%%%%%%%%%%%%%%%%%%%%%%%%%%%%%%%%%%%%%%%%%%%%%%%%%%%%%%%%%%%%%%%%%%%%

\begin{abstract}
We study the  existence of positive  solutions with prescribed $L^2$-norm for
the mass supercritical Schr\"odinger equation
$$
-\Delta u+\lambda u-V(x)u=|u|^{p-2}u\qquad u\in H^1(\R^N),\quad \lambda\in \R,
$$
where $V\ge 0$,   $N\ge 1$ and $p\in\left(2+\frac 4 N,2^*\right)$, $2^*:=\frac{2N}{N-2}$ if $N\ge 3$ and $2^*:=+\infty$ if $N=1,2$.
We treat two cases. Firstly, under an explicit smallness assumption on $V$ and no condition on
the mass, we prove the existence of a mountain pass solution at positive energy level, and we
exclude the existence of solutions with negative energy. Secondly, requiring that the mass is smaller than some explicit bound, depending on $V$, and that $V$ is not too small in a suitable sense, we find two solutions: a local minimizer with negative energy,
and a mountain pass solution with positive energy. Moreover, a nonexistence result is proved.
\end{abstract}


\bigskip


\sezione{Introduction}

We consider the problem
$$
\left\{
\begin{array}{l}
-\Delta u-V(x)u +\lambda u=u^{p-1}\smallskip\\
 \lambda\in\R,\quad u\in S_\rho,\quad u\ge 0,
 \end{array}
\right.
\eqno{(P_\rho)}
$$
where $N\ge 1$ and
$$
S_\rho=\left\{u\in H^1(\R^N)\ :\ \int_{\R^N}u^2dx=\rho^2\right\},\qquad \rho>0\,.
$$
Throughout the paper we assume
\begin{equation}\label{eq:base_ass}
2+\frac 4 N < p < 2^*
\qquad\text{ and }\qquad
V \ge 0,\ V\not\equiv0,
\end{equation}
up to further restrictions on some Lebesgue norm of the measurable potential $V$ (as usual $2^*:=\frac{2N}{N-2}$ if $N\ge 3$ and $2^*:=+\infty$ if $N=1,2$).

Problems of the form $(P_\rho)$ come from the study of standing waves for the nonlinear Schr\"odinger  equation
 \beq
 \label{S}
 iw_t+\Delta w+V(x)w=f(w), \qquad \mbox{ in } \R^N\times (0,\infty),
 \eeq
that is solutions of the form
 \beq
 \label{A}
 w(x,t)=e^{i\lambda t}u(x), \qquad (x,t)\in\R^N\times (0,\infty)
 \eeq
where $u$ is a real function. Here we consider the model case of a power nonlinearity $f(w)=|w|^{p-2}w$.

A lot of efforts have been done studying problem (\ref{S}), \eqref{A} for a fixed frequency $\lambda\in\R$; the literature in this direction is huge and we do not even make an attempt to
summarize it here (see  e.g. \cite{AM,Cmilan,S} for a survey on almost classical results, the recent papers \cite{DS,MP21} and references therein for new contributions).

On the other hand, a significant point of view consists in considering problem \eqref{S} when the mass of the particle is known, i.e.  the real
function $u$ in \eqref{A} is in $S_\rho$, $\rho>0$ fixed,
and $\lambda$ is an unknown of the problem. A natural approach to such a question is by variational methods. Indeed,   when $V\in L^r(\R^N)$, for some $r\in [N/2,+\infty]$,  $r\ge1$ ($r>1$ if $N=2$), solutions of $(P_\rho)$ can be found as critical points of the energy functional
\beq
\label{def_F}
F(u)=\frac12\int_{\R^N}(|\D u|^2-V(x)u^2)dx-\frac1p \int_{\R^N}|u|^{p}dx\qquad u\in H^1(\R^N),
\eeq
constrained on $S_\rho$.   Here $\lambda$ comes out as a Lagrange multiplier.
It turns out that the behavior of $F$ on  $S_\rho$ changes according to the value of $p$, in relation with the mass critical exponent $p=2+\frac4N$.

In the so called mass-subcritical case $p\in (2,2+\frac 4N)$, the problem can
be faced by minimization. In this respect, among many results, we refer the
reader to the classical paper \cite{PLL}, and to the recent paper \cite{IM}, where the existence of global minimizers is obtained for more general
nonlinearities and for negative potentials, that is $V\ge 0$ in our framework.
Several related results about the minimization of the nonlinear Schr\"odinger energy have been obtained
on metric graphs, also in the mass-critical case \cite{AST1,AST2,PSV}.

On the other hand, in the mass-supercritical case $p>2+\frac 4N$, the functional $F$ is unbounded
from below on $S_\rho$, as it is readily seen testing the functional on
$u_h(x):=h^{\frac N2}\bar u(hx)$, $h>0$, for a fixed $\bar u\in S_\rho$. Hence
the problem cannot be solved by (global) minimization arguments, and even the
appropriate definition of ground state is not clear. This problem has been faced in the seminal paper by Jeanjean \cite{Je}, for general nonlinearities in the autonomous case (i.e. $V$ constant): the key idea in \cite{Je} is to obtain a
mountain pass solution on $S_\rho$, by exploiting a natural constraint related to the Pohozaev
identity. After \cite{Je}, only recently more general autonomous equations and systems have been considered, also
refining and developing this initial strategy, see \cite{BdV,BJS,BS1,BJ,GJ,BS2,IT,S1,S2} and
references therein. On the other hand, mass supercritical equations on bounded domains and/or with (step well)
trapping potentials, i.e. potentials as in $(P_\rho)$ satisfying
\[
\lim_{|x|\to+\infty} -V(x) = +\infty
\]
(or even $-V\equiv +\infty$ outside some bounded $\Omega$), have been considered in
\cite{NTV1,NTV2,PV,NTV3} (see also \cite{BBJV} for the case of partial confinement). In this case, the trapping nature of the
potential provides enough compactness to cause the existence of solutions which are
local minimizers of $F$ on $S_\rho$ also in the mass-supercritical case, at least when
$\rho$ is sufficiently small. On the other hand, for non-trapping potentials very few
results are available: in particular, weakly repulsive potentials, i.e.
\begin{equation}\label{eq:wr}
-V(x) \ge \liminf_{|x|\to+\infty} -V(x) > -\infty,
\end{equation}
were considered in the very recent paper \cite{BMRV}.
When \eqref{eq:wr} holds, then the mountain pass structure
by Jeanjean is destroyed, but in \cite{BMRV} a new variational principle exploiting the Pohozaev identity
can be used to obtain existence of solutions with high Morse index. To conclude this discussion
about the previous literature, since the results we have discussed so far are all of variational
nature, let us mention that also topological methods have been applied, for instance
in \cite{CV,BZZ,PPVV}, also in connection with ergodic Mean Field Games systems.

In this paper we consider a class of mass-supercritical, and Sobolev sub-critical, problems
with weakly attractive potential, that is
\[
-V(x) \le \limsup_{|x|\to+\infty} -V(x) < +\infty
\]
(notice that, up to subtracting the constant $\rho^2\limsup_{|x|\to+\infty} -V(x)$ in $F$, this
corresponds to assuming \eqref{eq:base_ass}). Under suitable assumptions, we obtain two
families of solution. On the one hand we show that the mountain pass solution of Jeanjean
still exists also in this situation; on the other hand, we show that such mountain pass
structure also provides a local minimizer, as in \cite{PV,NTV3}, even though the potential
is not trapping at all. More precisely, we first show that, under an explicit smallness
assumption on $V$ and no condition on the mass, a mountain pass solution at positive energy level
exists; under the same smallness assumption we also show that no solution at negative energy
values can exist. Secondly, requiring that the mass is smaller than some explicit bound,
depending on $V$, and that $V$ is not too small in a suitable sense, we find two solutions: a
local minimizer with negative energy, and a mountain pass solution with positive energy. Notably,
this second result holds true also in dimensions $N=1,2$.

In order to state our results we define the auxiliary function
\begin{equation}\label{eq:defW}
W(x) = V(x)|x|.
\end{equation}
\begin{teo}
\label{T1}
Let $N\ge 3$ and \eqref{eq:base_ass} hold true.
There exists a positive explicit constant $L=L(N,p)$ such that if
\begin{equation}\label{eq:ass_VW_senza_rho}
\max\{\|V\|_{N/2},\|W \|_{N}\}<L
\end{equation}
then $(P_\rho)$ has a mountain pass solution for every $\rho>0$, at a positive energy level,
while no solution with negative energy exists.
\end{teo}
\begin{teo}\label{Tmin}
Let $N\ge1$ and \eqref{eq:base_ass} hold true, and let  $r\in\left(\max(1,\frac{N}{2}),+\infty\right]$,
$s\in\left(\max(2, {N}),+\infty\right]$.
\begin{enumerate}
\item There exist positive explicit constants $\sigma=\sigma(N,p,r)$ and $K=K(N,p,r)$ such that, if
\begin{eqnarray}
&&  \text{either $r<+\infty$, or $r=+\infty$ and }\lim_{|x|\to+\infty}V(x)=0,\\
\label{eq:main_ass_MPG} && \|V\|_{r}\cdot \rho^\sigma  < K\text{ and}\\
\label{eq:V_neg_spect} && \text{there exists }\varphi\in S_\rho :
\int_{\R^N}(|\D \varphi|^2-V(x)\varphi^2)\,dx \le0,
\end{eqnarray}
then $(P_\rho)$ has a solution, which corresponds to a local minimizer of $F$ on $S_\rho$ with negative energy.
\item There exist positive explicit constants $\sigma_i=\sigma_i(N,p,r)$, $\bar\sigma_i=\bar\sigma_i(N,p,s)$, $i=1,2$,
and $\tilde L=\tilde L(N,p,r,s)$ such that, if
\begin{equation}\label{eq:ass_VW_con_rho}
\max\{\|V\|_{r}\cdot\rho^{\sigma_i},\|W \|_{s}\cdot\rho^{\bar\sigma_i}\}<\tilde L,\qquad i=1,2,
\end{equation}
then $(P_\rho)$ has a mountain pass solution at a positive energy level.
\end{enumerate}
\end{teo}
\begin{rem}
We point out that our results are not perturbative, indeed all the constants in the above
theorems can be made explicit with respect to the structural parameters $N$, $p$, $r$, $s$,
 see e.g.
\eqref{ipotesi V1}--\eqref{ipotesi V4}, or \eqref{eq:L1}--\eqref{eq:L3}.
\end{rem}
\begin{rem}
Under the assumption of Theorem \ref{Tmin}, it is standard to prove that the non-empty set
of local minimizers is (conditionally) orbitally stable for the associated evolution equation.
\end{rem}
\begin{rem}\label{rem:bottom}
It is well known that a sufficient condition for \eqref{eq:V_neg_spect} to hold true
is that
\begin{equation}\label{eq:suff_cond_bottom}
\inf_{B_R} V \ge \eta,\quad\text{with }
\begin{cases}
\eta>0,\ R>0 & N=1,2\smallskip\\
R^2\eta>N(N-2) & N\ge3.
\end{cases}
\end{equation}
In particular, no condition is required in dimension $N=1,2$ as long as $V\ge0$, $V\not\equiv0$.
See also \cite[Theorem B (i)]{IM} and Lemma \ref{lem:negativebottom} for further details.
\end{rem}
\begin{rem}
Notice that, for every $V$ (either satisfying \eqref{eq:V_neg_spect} or not),
it is always possible to choose $\rho$ sufficiently small so that
both \eqref{eq:main_ass_MPG} and \eqref{eq:ass_VW_con_rho} hold true.
Moreover, if $N\ge 3$, assumption   \eqref{eq:suff_cond_bottom} just requires
\[
\|V\|_r^r \ge \eta^r |B_R| \ge \eta^{r-\frac{N}{2}} \frac{\omega_N}{N} \left[N(N-2)
\right]^{\frac{N}{2}}.
\]
In particular, since $r>\frac{N}{2}$, one can exhibit potentials with arbitrarily small
$L^r$ norm fulfilling the assumptions of Theorem \ref{Tmin} (with sufficiently small $\rho$ and
large $R$).
\end{rem}
\begin{rem}
The mountain pass geometry in Theorems \ref{T1} and \ref{Tmin} is essentially
the same, therefore most probably the mountain pass solutions coincide. As
a matter of fact, the explicit dependence on $r,s$ show that $\sigma_1=\sigma_2=0$
if $r=N/2$ and $s=N$, so that in this case \eqref{eq:ass_VW_con_rho} and
\eqref{eq:ass_VW_senza_rho} coincide. On the other hand, if $r=N/2$ then also
$\sigma =0$, and also \eqref{eq:main_ass_MPG} reduces to \eqref{eq:ass_VW_senza_rho}.
Nonetheless, in this case \eqref{eq:ass_VW_senza_rho} and \eqref{eq:V_neg_spect}
are not compatible, so that the minimizer with negative energy does not exist.
\end{rem}
\begin{rem}
In principle, for our results we only need $V$, $W$ to belong to suitable Lebesgue spaces.
On the other hand, if we also have $V\in C^{0,\alpha}_{\loc} (\R^N)$
then all the solutions we find are classical and, by the strong maximum principle,
they are strictly positive in $\R^N$.
\end{rem}
\begin{rem}
The main difficulty to prove the existence of the mountain pass solution is the analysis of the behavior of a bounded Palais-Smale sequence related to the mountain pass level.
To overcome this difficulty, we prove that the Lagrange multiplier associated our PS-sequence is positive and then we use an almost classical
splitting result in the unconstrained sub-critical case (see \cite{BC}).
\end{rem}

\vspace{2mm}

Observe that in \cite{BMRV} an existence result is proved when the potential $V$ in $(P_\rho)$ is negative, while in the present paper we find solutions for nonnegative $V$.
In the following proposition we give another contribution in the study of the problem with a nonexistence result, analogous to \cite[Theorem 1.1]{CM}.
\begin{prop}
\label{TNE}
Let $p\in (2,2^*)$, $V\in   L^\infty(\R^N)$  and assume that there exists  $\frac{\partial V}{\partial\nu}\in L^s(\R^N)$ for some $\nu\in\R^N\setminus\{0\}$ and $s\in [\max(1,\frac{N}{2}) ,+\infty]$. If $\frac{\partial V}{\partial\nu}\ge 0$ and $\frac{\partial V}{\partial\nu}\not\equiv 0$, then problem
 \beq
 \label{Rs}
 -\Delta u+\lambda u-V(x)u=|u|^{p-2}u\qquad u\in S_\rho, \quad \lambda\in\R
\eeq
 has no solutions in $C^1(\R^N)\cap W^{2,2}(\R^N)$.
\end{prop}

\begin{rem}
It is worthwhile noticing that our results are compatible with Proposition \ref{TNE}. 
We have only to remark that, if we consider  $r=+\infty$ in Point 1 of Theorem \ref{Tmin}, 
then $V$ is explicitly required to vanish at infinity (that is not allowed by Proposition 
\ref{TNE}) and that also in the case $r=s=+\infty$ in Point 2 the potential $V$ vanishes at 
infinity, because
\beq
\label{1823}
|V(x)|\le\frac{ \|W\|_\infty}{|x|}.
\eeq
\end{rem}

Another result that could have some interest, for example in stating constraints to work with, is the following necessary condition for critical points of $F$ on $S_\rho$.
\begin{prop}
\label{CNC}
Let $p\in (2,2^*)$ and $V\in  L^r(\R^N)$ for some $r\in [\max(1,\frac{N}{2}),+\infty]$.
If $u\in S_\rho\cap W^{2,2}(\R^N)$ is a critical point for $F$ constrained on $S_\rho$, then
$$
\int_{\R^N}V(x)\,u\,\frac{\partial u}{\partial\nu}\, dx=0
$$
for every direction $\nu\in \R^N\setminus\{0\}$.
\end{prop}

Before concluding this introduction, we want to observe that in exterior domains or in some domains with unbounded boundary the splitting Lemma \ref{splitting-lemma} holds true again. Moreover, one can analyze the displacement of the potential and the mass $\rho$ in order to recover the mountain pass geometry. Hence, it would be interesting to investigate whether the mountain pass solution exists, taking into account that in our approach the Pohozaev identity plays a crucial role. We refer the reader to \cite{MM06,M04,MP98} for existence results in this framework, in the unconstrained case.

The paper is organized as follows: in Section 2 we introduce the main notations and some preliminary results,  and prove Proposition  \ref{TNE}, in Section 3 we prove Theorem \ref{T1} while Section 4 is devoted to the proof of Theorem \ref{Tmin}.

\section{Notation and preliminary results}

For $p\in [1,+\infty]$, we denote by $L^p(\rn)$ the Lebesgue's space
with norm $\|\cdot\|_p$ and by
$H^1(\rn)$, $D^{1,2}(\R^N)$  the usual Sobolev spaces with the norm $\|\cdot\|_H$ and $\|\D u\|_2$,
respectively;
$S$ will denote the Sobolev constant, namely:
\begin{equation}\label{cost sobolev}
S=\inf_{u\in H^1(\R^N)\setminus\{0\}}\frac{\|u\|^2_H}{\|u\|^2_{2^*}}=\inf_{u\in D^{1,2}(\R^N)\setminus\{0\}}\frac{\|\D u\|^2_2}{\|u\|^2_{2^*}};
\end{equation}
$c,C$ are constants which may vary from line to line (structural constants will depend on $N$, $p$, $r$, $s$, while
the dependence on $V$, $W$ and $\rho$ will be made explicit whenever useful).
We fix the constant
\begin{equation}\label{gamma}
\gamma=\frac N 2(p-2)>2.
\end{equation}

%%%%%
We recall that for every $\rho>0$
there exists a unique  solution $Z_\rho$, up to translations, for the limit problem
 \begin{equation}\label{burbuja-rho}
\left\{
\begin{aligned}
  &-\Delta Z_\rho+\la_\rho Z_\rho=Z_\rho^{p-1}\\
  &Z_\rho\in S_\rho,\ Z_\rho>0 ,
\end{aligned}
\right.
\end{equation}
with $\lambda_\rho>0$.
The function $Z_\rho$ is radial and it is a mountain pass
critical point for
$$
F_\infty (u)=\frac12\int_{\R^N}|\D u|^2dx-\frac1p\int_{\R^N}|u|^p dx\qquad u\in H^1(\R^N)
$$
constrained on $S_\rho$.
%%%%%

By scaling, $Z_\rho$ can be expressed in terms of the unique positive solution  $U\in H^1(\R^N)$  of
\[
\left\{
\begin{aligned}
 & -\Delta U+U = U^{p-1}\\
 & U>0,\ U(0) = \|U\|_\infty.
\end{aligned}
\right.
\]
More precisely, setting $\rho_0=\|U\|_2$, for $\rho>0$ we define:
$$
\mu_{ \rho}=\left(\frac{\rho}{\rho_0}\right)^{\frac{2(p-2)}{N (p - 2) - 4}}\,,
%\qquad q=\frac{4N-2p(N-2)}{N(p-2)-4}>0\,,
$$
 then in \eqref{burbuja-rho} we have
 \beq
\label{Zrho}
Z_\rho(x)=\mu_{  \rho}^{-\frac{2}{p-2}}U(x/\mu_{ \rho})\,,\qquad
\la_\rho=\mu_{\rho}^{-2}=\left(\frac{\rho_0}{\rho}\right)^{\frac{4(p-2)}{N(p-2)-4}}>0\,.
\eeq
Setting $m_\rho=F_\infty (Z_\rho)$, so that $m_{\rho_0}=F_\infty(U)$, observe that
\begin{equation}\label{eq:poho}
  m_\rho=m_{\rho_0}\left(\frac{\rho_0}{\rho}\right)^{ \frac{4N-2p(N-2)}{N(p-2)-4}} =
  \frac{N(p-2)-4}{4N-2p(N-2)}\left(\frac{\rho_0}{\rho}\right)^{\frac{4(p-2)}{N(p-2)-4}}\,\rho^2 =
  \frac{N(p-2)-4}{4N-2p(N-2)}\lambda_\rho\,\rho^2
  %= \dfrac{1}{q}\,\lambda_\rho\rho^2\,
  .
\end{equation}
In the following lemma we recall the well-known
Gagliardo-Nirenberg inequality.
\begin{lemma}
For every $N\ge3$, $2\le q\le 2^*$ there exists $G_q>0$, depending on $N$ and $q$, such that
\begin{equation}\label{gagliardo_q}
\|u\|_q\leq
G_q \|u\|_2^{1-\frac{N(q-2)}{2q}}\|\n u\|_2^{\frac{N(q-2)}{2q}}
\end{equation}
for every $u\in H^1(\rn)$. The inequality holds true also in $N=1,2$, for every $2\le q<+\infty$.

In particular, if $q=p$:
\begin{equation}\label{gagliardo}
\|u\|_p\leq
G \|u\|_2^{1-\frac \gamma p}\|\n u\|_2^\frac \gamma p
\end{equation}
where  $\gamma$ is defined in \eqref{gamma} and $G=G_p$.
\end{lemma}
Of course, if $N\ge3$, the above inequality holds true also for $q=2^*$, reducing to
the first Sobolev inequality in \eqref{cost sobolev} (with $S=G_{2^*}^{-2}$).
It is well known, see \cite{We}, that $G$ is achieved by $Z_\rho$, for any $\rho$.
Recalling \eqref{eq:poho}, standard calculations (see e.g. the appendix in \cite{BMRV}) yield
\begin{equation}\label{def G}
G=\frac{\|U\|_p}{\|U\|_2^{1-\frac \gamma p}\|\n U\|_2^\frac \gamma p}
= \frac{(2p)^{\frac1p}}{(2N-p(N-2))^{\frac{p-\gamma}{2p}}(N(p-2))^{\frac{\gamma}{2p}}}\left(\frac{N(p-2)-4}{2}\right)^{\frac{p-2}{2p}}
m_{\rho_0}^{-\frac{p-2}{2p}}.
\end{equation}

Next, we recall some basic estimates involving $V$ and $W$ (defined in
\eqref{eq:defW}), that immediately follow from H\"older,  Gagliardo-Nirenberg and Sobolev inequalities.
\begin{lemma}\label{stime_q V W}
For every $2\le q< 2^*$ we have
\begin{eqnarray*}\label{stima_q V3 eq}
\left|\ir V(x)u^2\, dx\right| &\leq& \|V\|_{\frac{q}{q-2}}\|u\|_{q}^2\leq G_q^2 \|V\|_{\frac{q}{q-2}}
\|u\|_2^{2-\frac{N(q-2)}{q}}\|\n u\|_2^{\frac{N(q-2)}{q}}
,
\\
\label{stima_q W N}
\left|\ir V(x)u(x)\n u(x)\cdot x\, dx\right|&\leq&
 \|W\|_{\frac{2q}{q-2}} \|u\|_q \| \nabla u\|_2
  \leq
 G_q \|W\|_{\frac{2q}{q-2}} \|u\|_2^{1-\frac{N(q-2)}{2q}}\|\n u\|_2^{1+\frac{N(q-2)}{2q}} .
\end{eqnarray*}

Furthermore, if $N\ge 3$, then the above inequalities hold also for $q=2^*$:
\begin{eqnarray}\label{stima V3 eq}
\left|\ir V(x)u^2\, dx\right| &\leq& \|V\|_{\frac N 2}\|u\|_{2^*}^2\leq S^{-1} \|V\|_{\frac N 2}\|\n u\|_2^2\,,\\
%\end{equation}
%\begin{eqnarray}
\label{stima W N}
%\nonumber
  \left|\ir V(x)u(x)\n u(x)\cdot x\, dx\right|&\leq&
%  \left(\ir W^2 u^2\, dx\right)^{\frac 1 2}\|\n u\|_2 \leq
%  \left(\|W^2\|_{\frac N 2}\|u^2\|_{\frac{N}{N-2}}\right)^{\frac 1 2}\|\n u\|_2\\
%   &=&
  \|W\|_N\|u\|_{2^*}\|\n u\|_2\leq S^{-1/2}\|W\|_N\|\n u\|_2^2\,.
\end{eqnarray}
\end{lemma}

Let $\lambda\in \R$, we denote by $I_\lambda,I_{\infty,\lambda}:H^1(\R^N)\to\R$ the functionals
$$
I_\lambda(u)=\frac12\int_{\R^N}|\D u|^2dx+\frac{\lambda}{2}\int_{\R^N}u^2dx
-\frac12\int_{\R^N}V(x) u^2dx-\frac1p\int_{\R^N}|u|^p dx,
$$
$$
I_{\infty,\lambda} (u)=\frac12\int_{\R^N}|\D u|^2dx+\frac{\lambda}{2}\int_{\R^N}u^2dx
-\frac1p\int_{\R^N}|u|^p dx.
$$
By Lemma \ref{stime_q V W} the functionals $I_\lambda$ and $F$ (see \eqref{def_F}) are well defined and of class $C^1$.

\proof[Proof of Proposition \ref{TNE}]
 Assume, by contradiction, that \eqref{Rs} has a solution $(u,\lambda)\in S_\rho\times\R$.
 Then $u\not\equiv 0$ and is a critical point of $I_\lambda$ on $H^1(\R^N)$.
Taking into account that $t\mapsto u(\cdot+t\nu)$ is a smooth curve in $H^1(\R^N)$, we get
$$
0=\frac{d\,}{dt}I_\lambda (u(x+t\nu))_{|_{t=0}}=\frac12\int_{\R^N} \frac{\partial V}{\partial\nu}\,u^2\, dx.
$$
Since $\frac{\partial V}{\partial\nu}\ge 0$ and $\frac{\partial V}{\partial\nu}\not\equiv 0$,  a contradiction arises once we verify $\meas\{x\in\R^N\ :\ u(x)=0\}=0$.
In order to show it, let us observe that $u$ is a non trivial solution of
$$
-\Delta u+c(x)u=0\qquad u\in C^1(\R^N)\cap W^{2,2}(\R^N),
$$
where $c(x)=\lambda-V(x)-|u|^{p-2}\in L^\infty_{\loc}(\R^N)$, so our claim follows from Theorem 1.7 in \cite{HS89}. 
\qed

\proof[Proof of Proposition \ref{CNC}]
Arguing as in the previous proof, for every direction $\nu\in\R^N\setminus\{0\}$ we have
$$
\int_{\R^N}V(x)\,u\,\frac{\partial u}{\partial\nu}\, dx=-\frac{d\,}{dt}I_\lambda (u(x+t\nu))_{|_{t=0}}=0.
$$
\qed

In order to get the compactness result for Palais-Smale sequences,
we  state a  Splitting Lemma for $I_\lambda$, in our framework. Its proof is very close to that given in \cite[Lemma 3.1]{BC}
for exterior domains, so we only  sketch it.

\begin{lemma}\label{splitting-lemma}
 Let us assume that
\begin{enumerate}
\item $N\ge 3$: $V\in L^{N/2}(B_1(0))$, $V\in L^{\tilde r}(\R^N\setminus B_1(0))$ for $\tilde r \in [N/2,+\infty],$
\item $N=1,2$: $V\in L^{r}(B_1(0))$, $V\in L^{\tilde r}(\R^N\setminus B_1(0))$ for $r,\tilde r \in (1,+\infty],$
\item in case $\tilde r=+\infty$, $V$ further satisfies $\lim_{|x|\to+\infty}V(x)=0$,
\item $\lambda>0$.
\end{enumerate}
If  $(v_n)_{n}$ is a bounded Palais-Smale sequence for $I_\la$ in $H^1(\R^N)$,
then, up to a subsequence,   $v_n$  weakly converge to a function  $v\in
H^1(\R^N)$ and   if the convergence is not strong then there exist an integer $k\ge  1$, $k$
nontrivial solutions
$w^1,\dots,w^k\in H^1(\R^N)$ to the limit equation
\beq
\label{1234}
-\Delta w+\la w=|w|^{p-2}w
\eeq
and $k$ sequences $\{y_n^j\}_n\subset\R^N$, $1\le j\le k$, such that
$|y_n^j|\to\infty$ as $n\to\infty$, $|y^{j_1}_n-y^{j_2}_n|\to\infty$,
  for $j_1\neq j_2$, as $n\to \infty$, and
\begin{equation}\label{splitting-PS}
v_n=v+\sum_{j=1}^k w^j(\cdotp-y^j_n)+o(1)\qquad\text{strongly in $H^1(\R^N)$.}
\end{equation}
Moreover, we have
\begin{equation}
\label{splitting-norm}
\|v_n\|_2^2=\|v\|_2^2+\sum_{j=1}^k \|w^j\|_2^2+o(1)
\end{equation}
and
\begin{equation}\label{splitting-energy}
I_\la(v_n)\to I_\la(v)+\sum_{j=1}^k I_{\infty,\la}(w^j)\qquad\text{as $n\to\infty$.}
\end{equation}
\end{lemma}

\begin{rem}\label{rem:ass_SL}
We notice that the assumptions on $V$ in this lemma follow from those in our main results.
Let us only observe that in the case $r=+\infty$ in Point 2 of Theorem \ref{Tmin},  from $\|W\|_s<\infty$ we get $\lim_{|x|\to\infty} V(x)=0$ if $s=+\infty$ (see \eqref{1823}) and
$$
\int_{\R^N\setminus B_1(0)}|V(x)|^sdx\le \int_{\R^N\setminus B_1(0)}|V(x)\cdot|x|~|^sdx\le \|W\|_s<\infty\quad\Longrightarrow\quad V\in L^s(\R^N\setminus B_1(0))
$$
if $s\in(\max(2,N),+\infty)$.
\end{rem}
\begin{proof}[Proof of Lemma \ref{splitting-lemma}]
In this proof we argue up to suitable subsequences. Let $v$ be the weak limit of $v_n$ and set $v_{1,n}:=v_n-v$. Then,  $ v_{1,n}\to 0$ weakly in $H^1(\R^N)$, strongly in $L^2_{\loc}(\R^N)$, $L^p_{\loc}(\R^N)$, and a.e. in $\R^N$. Moreover
\beq
\int_{\R^N}V(x)v_{1,n}^2dx = \int_{B_1(0)}V(x)v_{1,n}^2dx  + \int_{\R^N\setminus B_1(0)}V(x)v_{1,n}^2dx = I + II.
\eeq
Assume first $N\ge 3$ and $\tilde r\in [N/2,+\infty)$, and set $\tilde r'$ the conjugate exponent of $r$.
We deduce by Egorov's Theorem that  $v_{1,n}^2\to0$ weakly in
$L^{N/(N-2)}(B_1(0))$ and in $L^{\tilde r'}(\R^N\setminus B_1(0))$, because $v_n^2$ is bounded in $L^{N/(N-2)}(B_1(0))$ and in $L^{\tilde r'}(\R^N\setminus B_1(0))$ and goes to $0$ a.e..
Hence both $I$ and $II$ converge to zero as $n\to\infty$.
Also in the case $\tilde r=+\infty$ the addendum $II$ goes to zero, because $v_{1,n}^2\to0$  in $L^2_{\loc}(\R^N)$, $\|v_{1,n}\|_2$ is bounded and $\lim_{|x|\to+\infty}V(x)=0$.
If $N=1,2$ the arguments above again prove that $I$ and $II$  converge to zero as $n\to\infty$.
%Now, if $R$ is sufficiently large, $II$ can be made arbitrarily small using an argument similar to
%Lemma \ref{stime_q V W}, the boundedness of $v_n$ and the fact that either
%$V\in L^{\tilde r}(\R^N\setminus B_1(0))$ for $\tilde r \in [N/2,+\infty)$,
%or $\lim_{|x|\to+\infty}V(x)=0$. On the other hand, as $v_n^2$ is bounded in $L^{N/(N-2)}(B_R)$
%and goes to $0$ a.e., by Egorov's Theorem we deduce that $v_n^2\to0$ weakly in
%$L^{N/(N-2)}(B_R)$, so that $I$ is arbitrarily small too.
Summing up, in any case we have that
\beq
\label{1200}
\int_{\R^N}V(x)v_{1,n}^2dx\longrightarrow 0,\text{ as }n\to\infty,
\eeq
and the sequence $\tilde v_{1,n}$ turns out to be a PS sequence for $I_{\infty,\lambda}$.

If $v_{1,n}\to 0$ in $H^1(\R^N)$ we are done, otherwise we can assume that $\|v_{1,n}\|_H\ge d_1$ for a suitable constant $d_1>0$. As a consequence, we deduce the existence of a constant $\tilde d_1>0$ such that
$$
\|v_{1,n}\|_p\ge \tilde d_1\qquad\forall n\in\N.
$$
Indeed, suppose by contradiction that $\|v_{1,n}\|_p\to 0$. Then, since $v_{1,n}$ is a bounded PS sequence for $I_\lambda$ and taking into account \eqref{1200}, we get
$$
\|\D v_{1,n}\|_2^2+\lambda\|v_{1,n}\|_2^2=\|v_{1,n}\|^p_p+o(1)=o(1),
$$
contrary to $\|v_{1,n}\|_H\ge d_1$.

Now, let  $\{Q_i\}_{i\in\N}$ be a decomposition of $\R^N$ by unitary dyadic cubes, and set
$$
l_n=\max_{i\in \N}\|v_{1,n}\|_{L^{p}(Q_i)}.
$$
Then there exists a constant $l>0$ such that $l_n\ge l$, for all $n\in\N$, because
\begin{eqnarray*}
\nonumber
0<\tilde d_1 \le \|v_{1,n} \|_{p}^{p}& =
&\sum_{i=1}^\infty\|v_{1,n}\|^{p}_{L^{p}(Q_i)}\\
&\le &l_n^{p-2}\sum_{i=1}^\infty\|v_{1,n}\|^2_{L^{p}(Q_i)}\le
c_1\, l_n^{p-2}\sum_{i=1}^\infty\|v_{1,n}\|^2_{H^1(Q_i)}\\
\nonumber
&\le & c_2\, l_n^{p-2},
\end{eqnarray*}
for suitable positive constants $c_1,c_2$ depending only on the Sobolev constant and the upper bound  of $\|v_n\|_H^2$.
 Let $y^1_n$ be the center of a cube $Q_{i_n}$ such that $d_n=\|v_{1,n}\|^{p}_{L^{p}(Q_{i_n})}$ and observe that $|y^1_n|\to \infty$, by $v_{1,n}\to 0$ in $L^p_{\loc}(\R^N)$.
Setting
$$
\tilde v_{1,n}=v_{1,n}(\cdot+y^1_n),
$$
it turns out that $\tilde v_{1,n}$ is a bounded PS sequence for $I_{\infty,\lambda}$.
So, $\tilde v_{1,n}\to w^1$ weakly in $H^1(\R^N)$ and in $L^p(\R^N)$,   in $L^p_{\loc}(\R^N)$ and a.e. in $\R^N$, where  $w^1$ is a weak solution of \eqref{1234}, non trivial because $\|w^1\|_{L^p(B_{\sqrt{N}}(0))}\ge l>0$.
Moreover,  in view of \eqref{1200},
$$
v_n=v+v_{1,n}=v+\tilde v_{1,n}(\cdot -y^1_n) =v+ w^1(\cdot -y^1_n)+[\tilde v_{1,n}(\cdot -y^1_n)-w^1(\cdot -y^1_n)],
$$
$$
\|v_n\|^2_{H}=\|v\|^2_H+ \|v_{1,n}\|^2_H+o(1)=\|v\|^2_H+ \|w^1\|_H^2 +\|\tilde v_{1,n}-w^1\|_H^2+o(1),
$$
$$
I_\lambda(v_n)=I_\lambda(v)+I_{\infty,\lambda}(v_{1,n})+o(1)=
I_\lambda(v)+I_{\infty,\lambda}(w^1)+I_{\infty,\lambda}(\tilde v_{1,n} -w^1 )+o(1).
$$
Iterating the procedure,  taking into account that $v_n$ is bounded and that the action of the ground state solution is positive, the proof is completed (see also \cite[Lemma 3.2]{MP98} for more details).
\end{proof}

Finally, we recall the following well-known fact, see e.g. \cite[Appendix A]{BMRV}.
\begin{lemma}\label{lem:append}
Let $w\in H^1(\R^N)$ be a non-trivial solution of
\[
-\Delta w+\la w=|w|^{p-2}w,
\]
for some $\lambda>0$. Then
\[
\lambda\ge \lambda_{\|w\|},
\qquad
F_\infty(w) \ge m_{\|w\|_2}>0,
\]
where $\lambda_\rho$ is defined in \eqref{Zrho} and $m_{\rho}$ in \eqref{eq:poho},
for every $\rho>0$.
\end{lemma}

%%%%%%%%%%%%%%%%%%%%%%%%%%%%%555
\section{Proof of Theorem \ref{T1}}
\label{PT1}

In this section we assume that $L>0$ is such that assumption
\eqref{eq:ass_VW_senza_rho} implies the following explicit bounds on $V$ and $W$,
for some fixed $\delta\in (0,1)$:
\begin{equation}\label{ipotesi V1}
\|V\|_{N/2}<(1-\delta)S;
\end{equation}
\begin{equation}\label{ipotesi V2}
N|4-p|S^{-1}  \|V\|_{\frac N 2}+4S^{-1/2}\|W\|_N<B,
\end{equation}
\begin{equation}\label{ipotesi V3}
\left[
AMN|4-p|+(N-2)D
\right]S^{-1} \|V\|_{\frac N 2}+
[4AM+2D]S^{-1/2} \|W\|_N
 <ABM
\,,
\end{equation}

where
\beq\label{def costanti}
\begin{array}{c}
 \vspace{2mm}
  A = [2N-(N-2)p],\qquad
 B = N(p-2)-4 ,\qquad
  D = N(p-2)^2 ,
  \\
 \displaystyle M = \left[\frac{\delta}{\gamma}\right]^{\frac{\gamma}{\gamma-2}}
  \left[
  \frac{\gamma}{2}-1
  \right]
  \left(
  \frac{p}{G^p}
  \right)^{\frac{2}{\gamma-2}}
  \frac{1}{m_{\rho_0}\rho_0^s},
\end{array}
\eeq
with $s=2\, \frac{2N-(N-2)p}{N(p-2)-4}$;
moreover:
\beq
\label{ipotesi V4}
3(p-4)^+ S^{-1}\|V\|_{N/2}+4S^{-1/2}\|W\|_N\le  N(p-2)-4.
\eeq
Notice that $(p-4)^+=0$ if $N\ge 4$.

We prove that $F$ has a mountain pass geometry, which ensures by Proposition \ref{bounded PS} the existence
of a Palais-Smale sequence.
Then, in order to recover compactness for this sequence,
we use the Splitting Lemma \ref{splitting-lemma},
and, to apply this, we need to prove that the limit
of the sequence of the Lagrange multipliers related to the PS-sequence is positive.

To start with, we focus on the geometric structure of $F$,
observing first the following scaling property.
For every $u\in S_\rho$ and $h>0$ we define the function $u_h\in S_\rho$ by
$$
u_h(x)=h^{\frac N 2}u(hx).
$$
Since $\n_x u_h(x)=h^{{\frac N 2}+1}\n_y u(hx)$, $y=hx$, we get
\begin{equation}\label{F su u riscalata}
F(u_h)=
\frac{h^2}{2}\ir|\n u|^2 dx
-\frac{h^{\frac N 2(p-2)}}{p}\ir|u|^p dx
-\ir V\left(\frac x h\right)u^2(x)dx.
\end{equation}

For fixed $u\in S_\rho$ we infer:
\begin{equation}\label{propr V uh 1}
\ir V(x)\, u_h^2(x)\, dx\le h^2\|V\|_{N/2}\|u\|_{2^*}^2\longrightarrow 0,\qquad\mbox{ as }h\to 0.
\end{equation}

Therefore, for every  $u\in S_\rho$ it follows:
\beq
\label{1031}
\lim_{h\to 0^+}\|\D u_h\|_2=0,\qquad
\lim_{h\to +\infty}\|\D u_h\|_2=\infty\, ,
\eeq
\begin{equation}\label{limiti Fuh}
\lim_{h\to 0^+}F(u_h)=0,\qquad
\lim_{h\to +\infty}F(u_h)=-\infty\,.
\end{equation}

The following lemma gives a lower estimate for $F$,
useful to prove that $F$ has a mountain pass geometry.

Assumption \eqref{ipotesi V1} and
inequalities \eqref{gagliardo} and \eqref{stima V3 eq}
imply
\begin{lemma}\label{Lstima1 F basso}
 \begin{equation}\label{stima1 F basso eq}
F(u)\geq \frac \delta 2\|\n u\|_2^2-c(\rho)\|\n
u\|_2^\gamma,\qquad\forall u\in S_\rho,
\end{equation}
where  $c(\rho)=\frac{G^p}{p}\rho^{p-\gamma}$, with $G$ defined in \eqref{def G}.
\end{lemma}
Hence from
\eqref{stima1 F basso eq} we infer that $\bar R>0$ exists such that
$$
\cM:=\inf\{F(u)\ :\ u\in S_\rho,\ \|\D u\|_2=\bar R\}>0.
$$
Now, let us consider the function $Z_\rho\in S_\rho$ introduced in \eqref{Zrho}.
By \eqref{1031} and \eqref{limiti Fuh} there exist $0<h_0<1<h_1$
 such that
\begin{eqnarray*}
&& \|\D (Z_\rho)_{h_0}\|_2<\bar R,\qquad F((Z_\rho)_{h_0})<\cM,\\
&& \|\D(Z_\rho)_{h_1}\|_2>\bar R,\qquad F((Z_\rho)_{h_1})<0.
\end{eqnarray*}
Then, we define in a standard way the mountain pass value
\beq
\label{MPv}
m_{V,\rho}:=\inf_{\xi\in\Gamma}\max_{t\in [0,1]}F(\xi(t))
\eeq
where
\beq
\label{1500}
\Gamma=\{\xi\in \cC^0([0,1];S_\rho ) \ :\ \xi(0)=(Z_\rho)_{h_0},\
\xi(1)=(Z_\rho)_{h_1}\}.
\eeq

\begin{rem}\label{mv-mr lemma}
{%\em
Since
$$
m_{\rho}=\inf_{\xi\in\Gamma}\max_{t\in [0,1]}F_\infty(\xi(t))
$$
(see \cite{Je}), it is immediately seen that
\begin{equation}\label{mv-mr 1}
m_{V,\rho}<m_\rho
\end{equation}
(it is sufficient to use the test path $\xi(t)=(Z_\rho)_{h_0(1-t)+h_1t}$
and use assumption \eqref{eq:base_ass}).
Moreover, it holds
\begin{equation}\label{mv-mr 2}
m_{V,\rho}\geq M m_\rho\qquad \forall\rho>0\,.
\end{equation}
In fact, if we set $f(t)=\frac \delta 2 t^2-c(\rho)t^\gamma$,
by \eqref{stima1 F basso eq} we infer $F(u)\geq f(|\nabla u|)$
and hence $m_{V,\rho}$ is greater than the maximum of $f$,
which is achieved for
$\bar t_\rho=\left( \frac{\delta}{\gamma\, c(\rho)} \right)^{\frac{1}{\gamma-2}}$,
getting
$$
f(\bar t_\rho)=
\delta^{\frac{\gamma}{\gamma-2}}
\left[
  \frac{1}{2\gamma^{\frac{2}{\gamma-2}}}-
  \frac{1}{\gamma^{\frac{\gamma}{\gamma-2}}}
  \right]
  \frac{1}{c(\rho)^{\frac{2}{\gamma-2}}}\,.
$$
Recalling that $c(\rho)=\frac{G^p}{p}\,\rho^{p-\gamma}$, by
\eqref{eq:poho} we obtain \eqref{mv-mr 2}.
}\end{rem}

In the following Lemma we recall a result, which can be directly derived
as a particular case of \cite[Theorem 4.5]{Ghou} and which is a key tool
in the proof of Proposition \ref{bounded PS}.

\begin{lemma}\label{lemma gossub}
Let $X$ be a Hilbert manifold and let $J\in C^1(X,\R)$ be a given functional.
Let $K\subset X$ be compact and consider a subset
$$
\cE\subset\{E\subset X:\,E\text{ is compact, $K\subset E$}\}
$$
which is invariant with respect to deformations leaving $K$ fixed. Assume that
$$
\max_{u\in K}J(u)<c:=\inf_{E\in\mathcal{E}}\max_{u\in E} J(u)\in\R.
$$
Let $\sigma_n\in \R$ be such that $\sigma_n\to 0$ and $E_n\in\mathcal{E}$ be a sequence such that
$$
c\le\max_{u\in E_n}J(u)<c+\sigma_n.
$$
Then there exists a sequence $v_n\in X$ such that
\begin{enumerate}
\item $c\le J(v_n)<c+ \sigma_n$,\label{propF}
\item $\|\nabla_X J(v_n)\|< \tilde{c}\sqrt{\sigma_n}$,\label{propF'}
\item $\dist(v_n,E_n)< \tilde{c}\sqrt{\sigma_n}$,\label{propv_n}
\end{enumerate}
for some constant $\tilde{c}>0$.
\end{lemma}

\begin{prop}\label{bounded PS}
There exists a Palais-Smale  sequence $(v_{n})_n$ for $F$ constrained
on  $S_\rho$ at the level $m_{V,\rho}$, namely
\begin{equation}\label{Palais-Smale-meps}
  F(v_n)\to m_{V,\rho},\qquad \nabla_{S_\rho}F(v_n)\to 0, \quad\mbox{ as }n\to\infty,
\end{equation}
such that
\begin{equation}
\label{Pohozaev-id}
\|\nabla
v_{n}\|_2^2-\frac{N(p-2)}{2p}\|v_{n}\|_{p}^{p}-\frac{1}{2}\int_{\R^N}V(x)(N
v_{n}^2+2v_n\nabla v_n\cdotp x)dx\to 0, \quad\mbox{ as }n\to\infty,
\end{equation}
\beq
\label{1731}
\lim_{n\to\infty} \|(v_n)^-\|_2=0.
\eeq
Moreover, the sequence $(v_{n})_n$ is bounded and the related Lagrange multipliers
\begin{equation}
\label{def-lambda_n}
\la_n:=-\frac{DF(v_n)[v_n]}{\rho^2}
\end{equation}
are bounded and verify, up to a subsequence, $\lambda_n\to \la$, with $\lambda>0$.
\label{prop-bounded-PS}
\end{prop}

\proof

\emph{Step 1. Existence of the Palais-Smale sequence}

The existence of a PS-sequence that verifies \eqref{Pohozaev-id} and \eqref{1731}
closely follows the arguments in \cite[Proposition 3.11]{BMRV},
where the authors adapt some ideas in \cite{Je}.
We recall the main strategy, referring to \cite{BMRV} for the details.
A key tool is to set:
$$
\wt{F}(u,h):=F(e^{\frac N2 h}u_h(e^h\cdot))\qquad\text{for all $(u,h)\in H^1(\R^N)\times \R$,}
$$
$$
\wt{\Gamma}:=\{\wt{\xi}\in \cC^0([0,1];S_\rho\times\R)\ : \ \wt{\xi}(0)=((Z_\rho)_{h_0},0),\,\wt{\xi}(1)=((Z_\rho)_{h_1},0)\}
$$
and
$$
\wt{m}_{V,\rho}:=\inf_{\wt{\xi}\in\wt{\Gamma}}\max_{t\in [0,1]}\wt{F}(\wt{\xi}(t)).
$$

It turns out that $\wt{m}_{V,\rho}=m_{V,\rho}$
and that,
if $(u_n,h_n)_n$ is a (PS)$_c$ sequence for $\wt{F}$ with $h_n\to 0$,
then $(u_n)_{h_n}$ is a (PS)$_c$ sequence for $F$.
Now, let us consider a sequence ${\xi}_n\in{\Gamma}$ such that
$$
m_{V,\rho}\le \max_{t\in [0,1]} {F}( {\xi}_n(t))< m_{V,\rho}+\frac{1}{n}.
$$
Then, we are in a position to apply  Lemma \ref{lemma gossub}
to $\wt{F}$ with
$$
X:=S_\rho\times\R,\quad K:=\{((Z_\rho)_{h_0},0),\,((Z_\rho)_{h_1},0)\},\quad
\mathcal{E}=\wt{\Gamma}_\rho,\quad
E_n:=\{({\xi}_n(t),0)\, :\,t\in[0,1]\}.
$$
As a consequence, there exist a sequence $(u_{n},h_{n})\in S_\rho\times\R$ and $\tilde{c}>0$ such that
$$
m_{V,\rho}-\frac{1}{n} < \wt{F}(u_{n},h_{n}) < m_{V,\rho}+\frac{1}{n}
$$
$$
\min_{t\in [0,1]}\|(u_{n},h_{n})-({\xi}_n(t),0)\|_{H^1(\R^N)\times\R}<\frac{\tilde{c}}{\sqrt{n}}
$$
$$
\|\nabla_{S_\rho\times\R}\wt{F}(u_{n},h_{n})\|<\frac{\tilde{c}}{\sqrt{n}}.
$$

We observe that, differentiating with respect to $h$,
we get the ``almost'' Pohozaev identity \eqref{Pohozaev-id},
while, differentiating with respect to the first variable on the tangent space to $S_\rho$,
we recover the equation \eqref{def-lambda_n} for the Lagrange multiplier.

\vspace{2mm}

\emph{Step 2. Boundedness of the Palais-Smale sequence}

We set
\beq\label{eq:def-a_n}
a_n:=\|\nabla v_n\|_2^2,\quad
b_n:=\|v_n\|_{p}^{p},\quad
c_n:=\int_{\R^N}V(x)v_n^2 dx,\quad
d_n:=\int_{\R^N}V(x)v_n\nabla v_n\cdotp x \, dx.
\eeq
By (\ref{Palais-Smale-meps}), (\ref{Pohozaev-id}) and (\ref{def-lambda_n})
we get
\begin{eqnarray}\label{PS-v_n}
a_n-c_n-\frac{2}{p}b_n= 2m_{V,\rho}+o(1)
\end{eqnarray}
 \begin{eqnarray}\label{lambda_n-new}
a_n-c_n+\la_n \rho^2= b_n+o(1)(a_n^{1/2}+1) \,
\end{eqnarray}
\begin{eqnarray}\label{Pohozaev-new}
a_n-\frac{N(p-2)}{2p}b_n-\frac{N}{2}c_n-d_n=o(1).
\end{eqnarray}
The term  $(a_n^{1/2}+1)$ is in \eqref{lambda_n-new}
because we do not know that $v_n$ is bounded in $H^1(\rn)$ yet.
By \eqref{PS-v_n}  and \eqref{Pohozaev-new}
we obtain
$$
\frac{N(p-2)-4}{2p}b_n= 2m_{V,\rho}-\frac{N-2}{2}c_n-d_n+o(1)\,
$$
and, recalling the definition of $B$ in \eqref{def costanti}, we infer:
\begin{equation}\label{a_n-bd}
\begin{aligned}
a_n&=\frac{4}{N(p-2)-4} \left(2m_{V,\rho}-
\frac{N-2}{2}c_n-d_n\right)+c_n+2m_{V,\rho}+o(1)\\
&=\frac{N(p-2)}{B}2m_{V,\rho}-
\frac{N(4-p)}{B}c_n-\frac{4}{B}d_n+o(1)\,.
\end{aligned}
\end{equation}

Since $m_{V,\rho}< m_\rho$ , $c_n\leq S^{-1}\|V\|_{\frac N 2}a_n$
and $|d_n|\leq S^{-1/2}\|W\|_N a_n$, we get
\beq
\label{a_n-bd 2}
  0\leq B a_n \leq
    N(p-2)2m_\rho +
   \left[N|4-p|S^{-1}\|V\|_{\frac N 2}+4S^{-1/2}\|W\|_N\right] a_n+o(1),
\eeq
so
\begin{equation}\label{a_n-bd 3}
\left(
B
-\left[N|4-p|S^{-1}\|V\|_{\frac N 2}+4S^{-1/2}\|W\|_N\right]
\right)a_n\leq
N(p-2)2m_\rho +o(1)
\,.
\end{equation}
By assumption \eqref{ipotesi V2}
we have $B
-\left[N|4-p|S^{-1}\|V\|_{\frac N 2}+4S^{-1/2}\|W\|_N\right] >0$, so that:
\begin{equation}\label{a_n-bd 4}
a_n\leq
\frac{N(p-2)2m_\rho}{B-\left[N|4-p|S^{-1}\|V\|_{\frac N 2}+4S^{-1/2}\|W\|_N\right] }+o(1)\,
\end{equation}
that yields the boundedness of the sequence $a_n$, and so of $b_n$, $c_n$, $d_n$ and $\lambda_n$.

\vspace{2mm}

\emph{Step 3. Positivity of the Lagrange multiplier}

\bigskip

By the previous step, we can assume that the sequences $a_n, b_n,c_n, d_n$ and $\lambda_n$ converge, up to a subsequence, to suitable $a,b,c,d$ and $\lambda$, respectively.

Since $c_n\leq S^{-1}\|V\|_{\frac N 2}a_n$ and $|d_n|\leq S^{-1/2}\|W\|_N a_n$,
recalling \eqref{mv-mr 2}, the bound on $a_n$ given by \eqref{a_n-bd 4} and the definitions
in \eqref{def costanti},
by \eqref{PS-v_n}--\eqref{Pohozaev-new}  we get:

\begin{equation}\label{lambda>0}
\begin{aligned}
\la\rho^2&=\frac{p-2}{p}b-2m_{V,\rho}\\
&=\frac{2(p-2)}{N(p-2)-4}\left(2m_{V,\rho}-\frac{N-2}{2}c-d\right)-2m_{V,\rho}\\
&=\frac{2N-(N-2)p}{N(p-2)-4}2m_{V,\rho}-
\frac{(N-2)(p-2)}{N(p-2)-4}c-\frac{2(p-2)}{N(p-2)-4}d\\
&\ge
\frac 1 B\Biggl\{
[2N-(N-2)p]M \\
&   -
\frac{N(N-2)(p-2)^2S^{-1}\|V\|_{\frac N 2} }
 {B-\left[N|4-p|S^{-1}\|V\|_{\frac N 2}+4S^{-1/2}\|W\|_N\right] }
  -
\frac{2N(p-2)^2S^{-1/2}\|W\|_N  }
{B-\left[N|4-p|S^{-1}\|V\|_{\frac N 2}+4S^{-1/2}\|W\|_N\right] }
\Biggr\}2m_\rho\\
&=\frac 1 B\left\{
AM-
\frac{\left[(N-2)S^{-1}\|V\|_{\frac N 2}+2S^{-1/2}\|W\|_N\right]D}
 {B-\left[N|4-p|S^{-1}\|V\|_{\frac N 2}+4S^{-1/2}\|W\|_N\right]}
\right\}2m_\rho\\
&=\frac 1 B\frac{ABM-\left\{
\left[
AMN|4-p|+(N-2)D
\right]S^{-1} \|V\|_{\frac N 2}+
[4AM+2D]S^{-1/2} \|W\|_N
\right\}}
{B-\left[N|4-p|S^{-1}\|V\|_{\frac N 2}+4S^{-1/2}\|W\|_N\right]}
2m_\rho
\end{aligned}
\end{equation}
and hence hypothesis \eqref{ipotesi V3} ensures the positivity of $\lambda$.

\qed

\begin{lemma}
\label{L>}
Let $v$ be a weak solution of $(P_\rho)$, for some $\rho>0$. If \eqref{ipotesi V4} holds  then
\beq
\label{F(v) positivo eq}
F(v)\ge 0.
\eeq
\end{lemma}
\begin{proof}
The function $v$ satisfies the Pohozaev identity:
\begin{equation}\label{Pohozaev v}
\frac 1 p\|v\|_p^p=
\frac{2}{N(p-2)}\|\n v\|_2^2-\frac{1}{p-2}\ir
V(x)v^2\, dx-\frac{2}{N(p-2)}\ir V(x)v\, \n v\cdot x\, dx
\end{equation}
%$$
%\ir |v|^p=\ir|\n v|^2+\ir\left(\lambda-V(x)\right)v^2
%$$
and we get:
\begin{eqnarray}\label{Fv}
 F(v)&=&\frac 1 2\|\n v\|_2^2-\frac 1 2\ir V(x)v^2\, dx-\frac 1 p\|v\|_p^p\\
\nonumber  &=&
  \left(\frac 1 2 -\frac{2}{N(p-2)}\right)\|\n v\|_2^2+
  \frac{4-p}{2(p-2)}\ir V(x)v^2\, dx+
  \frac{2}{N(p-2)}\ir V(x)v\n v\cdot x\, dx.
\end{eqnarray}
If $N\ge 4$ or $N=3$ and $p\in(\frac{10}{3},4]$ then, using \eqref{stima W N},
\eqref{Fv} gives
\begin{eqnarray*}
F(v) &\geq&
  \left(\frac 1 2-\frac{2}{N(p-2)}\right)\|\n v\|_2^2-\frac{2}{N(p-2)}S^{-1/2}\|W\|_N\|\n v\|_2^2\\
  &=&\left(\frac 1 2 -\frac{2}{N(p-2)}-\frac{2}{N(p-2)}S^{-1/2}\|W\|_N\right)\|\n v\|_2^2\,,
\end{eqnarray*}
so, since by assumption \eqref{ipotesi V4}
$$
\frac 1 2 -\frac{2}{N(p-2)}-\frac{2}{N(p-2)}S^{-1/2}\|W\|_N\ge 0,
$$
inequality \eqref{F(v) positivo eq} follows.

If $N=3$ and $p\in (4,6)$ then \eqref{Fv} gives
\begin{eqnarray*}
F(v) &\geq&
\left(\frac 1 2-\frac{2}{3(p-2)}\right)\|\n v\|_2^2-
\frac{p-4}{2(p-2)}S^{-1}\|V\|_{3/2}\|\D v\|_2^2
-\frac{2}{3(p-2)}S^{-1/2}\|W\|_3\|\n v\|_2^2\\
  &=&\left[\frac{3(p-2)-4}{6(p-2)}-\left(\frac{p-4}{2}S^{-1}\|V\|_{3/2}+\frac23S^{-1/2}\|W\|_3\right)\frac{1}{p-2}\right]\|\n v\|_2^2\,,
\end{eqnarray*}
and the claim follows from \eqref{ipotesi V4}.
\end{proof}

 \begin{proof}[End of the proof of Theorem \ref{T1}]
We have proved in Lemma \ref{L>} that no solution of $(P_\rho)$ with negative energy exists.
Now, let us prove the existence of a solution.
Let us consider the bounded Palais-Smale sequence $v_n$ given by Proposition \ref{bounded PS}.
Since $v_{n}$ is bounded in $H^1(\R^N)$,
there exists $v\in H^1(\R^N)$ such that, up to a subsequence,
$v_n$ converges weakly in $H^1(\R^N)$ and a.e. in $\rn$ to a function $v\in H^1(\R^N)$,
which turns out to be a weak solution  of
\begin{equation}\label{eq:v}
-\Delta v+(\la-V)v=|v|^{p-2}v
\end{equation}
with $\|v\|_2\le \rho$.
To prove the theorem we will show that actually $v_n$ converge to $v$ strongly in $H^1$.
Then we are done, because in such a case $\|v\|_2=\rho$ and  $v\ge 0$ by \eqref{1731}.

Since
$$
\int_{\R^N}\nabla v_n\nabla \varphi \,dx+\int_{\R^N}V(x)v_n \varphi \,dx-\int_{\R^N}|v_n|^{p-2}v_n\varphi \,dx =
-\la_n \int_{\R^N}v_n \varphi \,dx + o(1)\|\varphi\|
$$
for every $\varphi\in H^1(\R^N)$,
we have
$$
\int_{\R^N}\nabla v_n\nabla \varphi \,dx+\int_{\R^N}V(x)v_n \varphi \,dx-\int_{\R^N}|v_n|^{p-2}v_n\varphi \,dx =
-\la \int_{\R^N}v_n \varphi \,dx +
(\la-\la_n) \int_{\R^N}v_n \varphi \,dx+
o(1)\|\varphi\|
$$
and hence
$$
\int_{\R^N}\nabla v_n\nabla \varphi \,dx+\int_{\R^N}V(x)v_n \varphi \,dx-\int_{\R^N}|v_n|^{p-2}v_n\varphi \,dx =
-\la \int_{\R^N}v_n \varphi \,dx +
o(1)\|\varphi\|\,,
$$
because $v_n$ is bounded in $H^1(\R^N)$.
Therefore, $v_n$ is also a Palais-Smale sequence for $I_\la$ at level
$m_{V,\rho}+\frac{\la}{2}\rho^2$,
so that we can apply the Splitting Lemma \ref{splitting-lemma}, getting:
$$
v_n=v+\sum_{j=1}^k w^j(\cdotp-y^j_n)+o(1)\,,
$$
being  $w^j$ solutions of
$$
-\Delta w^j+\lambda w^j=|w^j|^{p-2}w^j
$$
and $|y^j_n|\to\infty$.
Assume by contradiction that $k\ge 1$ or, equivalently, that $\mu:=\|v\|_2<\rho$.
Recall that by \eqref{eq:poho} and by Lemma \ref{lem:append}
\begin{equation}\label{stima Fj basso}
 m_{\alpha}>m_{\beta} \qquad\text{if $0<\alpha<\beta$}\quad\mbox{ and }\quad F_\infty(w^j)\ge m_{\alpha_j},
\end{equation}
where $\alpha^j:=\|w_j\|_2$, $j\in\{1,\ldots,k\}$.
The condition $F(v_n)\to m_{V,\rho}$ and (\ref{splitting-energy}) implies
\begin{equation}\label{splitt eq}
m_{V,\rho}+\frac{\la}{2}\rho^2=
F(v)+\frac{\la}{2}\mu^2+\sum_{j=1}^k F_\infty(w^j)+\frac{\la}{2}\sum_{j=1}^k \alpha_j^2\,.
\end{equation}
By \eqref{splitting-norm} we have
$$
\rho^2=\mu ^2+\sum_{j=1}^k \alpha_j^2\,,
$$
and \eqref{splitt eq} becomes
\begin{equation}\label{splitt eq2}
m_{V,\rho}=F(v)+\sum_{j=1}^k F_\infty(w^j)\,.
\end{equation}

Using \eqref{F(v) positivo eq},\eqref{stima Fj basso}
and the fact that $\alpha_j\le\rho$,
we infer that the right-hand side of \eqref{splitt eq2} is  greater or equal
to $m_\rho$.
This contradict the fact that the left-hand side of \eqref{splitt eq2}
is strictly less than $m_\rho$ (see \eqref{mv-mr 1}), proving our result.
\end{proof}

\section{Proof of Theorem \ref{Tmin}}

\subsection{Existence of a local minimizer}

In order to prove the first part of Theorem \ref{Tmin}, we first show that, under
\eqref{eq:main_ass_MPG}, $F$ restricted on $S_\rho$ admits a mountain pass
structure which depends on $\|V\|_{r}$ but is uniform with respect to $\rho$. More precisely,
we have the following
\begin{prop}\label{prop:MPgeom}
Let $N\ge1$ and  $r\in\left(\max(1,\frac{N}{2}),+\infty\right]$.
There exist positive explicit constants $\sigma$, $K$, $\Theta$ and $\Upsilon$, only
depending on $N,p,r$, such that, if \eqref{eq:main_ass_MPG} holds true then
$$
\inf\{F(u)\ :\ u\in S_\rho,\ R_*-\eps\le\|\D u\|_2\le R_*\}>0
$$
where
\[
R_* = \Theta \cdot \|V\|_{r}^{\Upsilon}
\]
and $\eps>0$ is sufficiently small, depending only on (a bound from above on) $\rho$.
\end{prop}
To prove the proposition, we use the following elementary lemma.
\begin{lemma}\label{lem:elemlem}
Let $A,B,s,\alpha,\beta$ be positive parameters, with $\alpha\le1$, and define
\[
f_z(t) = t - A z^s t^{1-\alpha} - B z t^{1+\beta},\qquad z,t>0.
\]
Let $z_*$ and $t_*$ be defined as
%\[
%z_*^{\alpha + r\beta} = \left(\frac{\alpha}{B}\right)^{\alpha}\left(\frac{\beta}{A}\right)^{\beta}(\alpha+\beta)^{-\alpha - \beta},
%\qquad
%t_*^{\alpha + r\beta} = \left(\frac{\alpha}{B}\right)^{r}\frac{A}{\beta}(\alpha+\beta)^{1-r}.
%\]
\begin{equation}\label{eq:explicit1}
z_* = \left(\frac{\alpha}{B}\right)^{\frac{\alpha}{\alpha + s\beta}}\left(\frac{\beta}{A}\right)^{\frac{\beta}{\alpha + s\beta}}(\alpha+\beta)^{-\frac{\alpha + \beta}{\alpha + s\beta}},
\qquad
t_* = \left(\frac{\alpha}{B}\right)^{\frac{s}{\alpha + s\beta}}
\left(\frac{A}{\beta}\right)^{\frac{1}{\alpha + s\beta}}(\alpha+\beta)^{\frac{1-s}{\alpha + s\beta}}.
\end{equation}
Then
\begin{equation}\label{eq:elemlem1}
0<z<z_* \qquad\implies\qquad f_z(t_*)>0.
\end{equation}
\end{lemma}
\begin{proof}
By direct calculation, it follows that $f_{z_*}(t_*)=0$.
Then \eqref{eq:elemlem1} follows, as $f_z(\cdot)$ is
(pointwise) decreasing with respect to $z$.
\end{proof}

\begin{proof}[Proof of Proposition \ref{prop:MPgeom}]
Let
\[
u\in S_\rho,\qquad
\|\D u\|_2= R,\qquad
r = \frac{q}{q-2} \text{ (with $2\le q<2^*$)}.
\]
By Lemma \ref{stime_q V W} and (\ref{gagliardo}) we know that
\beq
\label{1055}
2 F(u) \ge R^2 - G_q^2 \|V\|_{\frac{q}{q-2}}
\rho^{2-\frac{N(q-2)}{q}} R^{\frac{N(q-2)}{q}} - \frac{2}{p}G_p^p
\rho^{p-\frac{N(p-2)}{2}}R^{\frac{N(p-2)}{2}}.
\eeq
Thus we can apply Lemma \ref{lem:elemlem}, with the corresponding notations,
writing
\begin{equation}\label{eq:explicit2}
t=R^2,\quad
z=\rho^{\frac{2N-p(N-2)}{2}},\qquad
A=G_q^2 \|V\|_{\frac{q}{q-2}},\quad
B=\frac{2}{p}G_p^p,
\end{equation}
and
\begin{equation}\label{eq:explicit3}
\alpha=\frac{2N-q(N-2)}{2q}\le1,\qquad
\beta=\frac{N(p-2)-4}{4},\qquad
s=\frac{2}{q}\cdot\frac{2N-q(N-2)}{2N-p(N-2)}.
\end{equation}
In particular, only $A$ depends on $\|V\|_{\frac{q}{q-2}}$, while $B,\alpha,\beta$ and $s$
just depend on $N,p$ and $r$ (via $q$). Then
\[
\alpha + s\beta = \frac{p-2}{q}\cdot\frac{2N-q(N-2)}{2N-p(N-2)},
\]
and we can write
\[
\rho_*^{\frac{2N-p(N-2)}{2}} = z_* = C(N,p,q)\cdot \|V\|_{\frac{q}{q-2}}^{-\frac{N(p-2)-4}{4}\cdot\frac{q}{p-2}\cdot\frac{2N-p(N-2)}{2N-q(N-2)}},
\]
in such a way that $z<z_*$ is equivalent to \eqref{eq:main_ass_MPG}, with a suitable choice
of $\sigma$ and $K$ (the choice is explicit, by combining \eqref{eq:explicit1}, \eqref{eq:explicit2} and \eqref{eq:explicit3}). Moreover
\[
R_*^2 = t_* = C(N,p,q)\cdot
\|V\|_{\frac{q}{q-2}}^{\frac{q}{p-2}\cdot\frac{2N-p(N-2)}{2N-q(N-2)}},
\]
and the proposition follows, with a suitable choice of $\Theta,\Upsilon$
(again, the choice is explicit, by combining \eqref{eq:explicit1}, \eqref{eq:explicit2} and
\eqref{eq:explicit3}).
\end{proof}

Hereafter we assume that $\rho,V$ satisfy \eqref{eq:main_ass_MPG}, that is
\[
0<\rho< \rho_* := H \cdot \|V\|_{r}^{-\tau},
\]
for suitable $H,\tau$.
By Proposition \ref{prop:MPgeom} we know that, for every $\alpha<\rho_*$,
\begin{equation}\label{eq:lem_rephrased}
\inf\{F(u)\ :\ u\in S_\alpha,\ \|\D u\|_2 = R_*\}>0,
\end{equation}
where $R_* = \Theta \cdot \|V\|_{r}^{\Upsilon}$ is independent of $\alpha$. We define
\begin{equation}\label{eq:def_c}
c_{V,\alpha}:= \inf\{F(u)\ :\ u\in S_\alpha,\ \|\D u\|_2\le R_*\}.
\end{equation}

The proof of the first part of Theorem \ref{Tmin} is based on the following proposition.
\begin{prop}\label{prop:exists_min}
If $
%\begin{equation}\label{eq:c<0}
c_{V,\rho}<0
%\end{equation}
$ then $c_{V,\rho}$ is achieved by a solution of $(P_\rho)$.
\end{prop}
In turn, the proof of the proposition is based on the following lemma.
\begin{lemma}\label{lem:ineq_c}
If $c_{V,\rho}<0$ then
\[
0< \alpha \le \rho \qquad\implies\qquad
c_{V,\alpha} \ge c_{V,\rho}.
\]
\end{lemma}
\begin{proof}
If $c_{V,\alpha} \ge0$ then the lemma is trivial. On the contrary, let $0>c'>c_{V,\alpha}$
and $u \in S_\alpha$ such that $\|\D u\|_2 < R_*$ and $F(u)<c'$. For every $t\ge 1$ we have
\[
tu \in S_{t\alpha} \quad\text{ and }\quad
F(tu) = \frac{t^2}{2}\left(\|\nabla u\|_2^2 - \ir V(x)u^2\, dx\right)-\frac{t^p}{2}\|u\|_p^p
\le t^2 F(u) <c'<0.
\]
We claim that $\left\|\nabla \frac{\rho}{\alpha} u\right\|_2\le R_*$. If not, there exists
$\bar t \in \left(1,\frac{\rho}{\alpha}\right)$ such that $\left\|\nabla \bar t u\right\|_2 =
R_*$, $\|\bar t u\|_2 < \rho$ and $F(\bar t u) <0$, in contradiction with
\eqref{eq:lem_rephrased}. Thus, by definition,
\[
c_{V,\rho} \le F\left( \frac{\rho}{\alpha} u\right) < c'.
\]
Since $c'>c_{V,\alpha}$ is arbitrary, the lemma follows.
\end{proof}
\begin{proof}[Proof of Proposition \ref{prop:exists_min}]
Let $(u_n)_n$ be a minimizing sequence for $c_{V,\rho}$.
By Proposition \ref{prop:MPgeom} we know that $\|\D u_n\|_2\le R_*-\eps$,
for some $\eps>0$ suitably small. Therefore, by Ekeland's principle,
we can assume that $(u_{n})_n$ is a Palais-Smale sequence for $F$ constrained
on  $S_\rho$, i.e.
\begin{equation}\label{Palais-Smale-min}
  F(u_n)\to c_{V,\rho},\qquad \nabla_{S_\rho}F(v_n) \to 0, \quad\text{ as }n\to\infty.
\end{equation}
Since both the functional and the constraint are even, we can choose each $u_n$ to be non-negative. Furthermore, $(u_n)_n$ is bounded by construction, and therefore also the Lagrange
multipliers
\[
\la_n:=-\frac{DF(u_n)[u_n]}{\rho^2}
\]
are bounded. Up to subsequences, we obtain that $u_n\rightharpoonup u\ge0$ weakly in $H^1(\R^N)$, and
$\lambda_n\to\lambda\in\R$. By \eqref{Palais-Smale-min},
\[
o(1) = \frac12\left( \|\nabla u_n\|_2^2 -\int_{\R^N} Vu_n^2\,dx - \|u_n\|_p^p\right) +\frac12 \lambda_n \rho^2 \le F(u_n)   +\frac12 \lambda_n \rho^2  = c_{V,\rho} + \frac12 \lambda \rho^2
+ o(1),
\]
which forces
\[
\lambda \ge - \frac{2 c_{V,\rho}}{\rho^2}>0.
\]
Arguing as in the proof of Theorem \ref{T1} we have that \eqref{Palais-Smale-min} implies that $(u_{n})_n$ is a (free) Palais-Smale sequence for the action functional
$I_\lambda$, with $\lambda>0$. Then the Splitting Lemma \ref{splitting-lemma} applies,
yielding
\[
u_n=u+\sum_{j=1}^k w^j(\cdotp-y^j_n)+o(1)\qquad\text{strongly in $H^1(\R^N)$,}
\]
where
\[
-\Delta u - Vu +\la u=u^{p-1},
\qquad
-\Delta w^j+\la w^j=|w^j|^{p-2}w^j,\ 1\le j\le k,
\]
and
\[
\|u\|_2^2+\sum_{j=1}^k \|w^j\|_2^2 = \rho^2,
\quad
I_\la(u)+\sum_{j=1}^k I_{\infty,\la}(w^j)=c_{V,\rho} + \frac12 \lambda \rho^2,
\quad
F(u)+\sum_{j=1}^k F_{\infty}(w^j)=c_{V,\rho}.
\]
Writing $\|u\|_2= \alpha\le\rho$, we have that $\|\nabla u\|_2 \le \liminf \|\nabla u_n\|_2 < R_*$, thus $F(u) \ge  c_{V,\alpha}$. Lemma \ref{lem:ineq_c} yields
\[
c_{V,\rho} = F(u)+\sum_{j=1}^k F_{\infty}(w^j) \ge c_{V,\alpha} + \sum_{j=1}^k F_{\infty}(w^j)
\ge c_{V,\rho} + \sum_{j=1}^k F_{\infty}(w^j),
\]
forcing
\[
\sum_{j=1}^k F_{\infty}(w^j)\le 0.
\]
By Lemma \ref{lem:append} we deduce that $k=0$, so that $u_n\to u$ strongly in $H^1(\R^N)$ and
the proposition follows.
\end{proof}
To show that $c_{V,\rho}<0$ we will use the fact that the bottom of the spectrum
of the operator $-\Delta - V$ is non-positive. As we mentioned, a sufficient
condition in this direction is contained in Remark \ref{rem:bottom}. For the
reader's convenience we sketch such result in the following lemma.
\begin{lemma}\label{lem:negativebottom}
Let $V\ge0$ satisfy \eqref{eq:suff_cond_bottom}. Then \eqref{eq:V_neg_spect} holds true.\end{lemma}
\begin{proof}
Let $B_R$ and $\eta$ be as in assumption \eqref{eq:suff_cond_bottom}, and without loss of
generality let us assume that $B_R$ is centered at $0$.

In case $N=1$ it is sufficient to choose, for a normalizing constant $t^*>0$,
\[
t^*\varphi(x)=
\begin{cases}
1 & |x|\le R\\
\left(\frac{kR - |x|}{(k-1)R}\right)^+   & |x|\ge R
\end{cases}
\]
with $k>1$ sufficiently large. In case $N=2$ it is sufficient to choose
\[
t^*\varphi(x)=
\begin{cases}
1 & |x|\le R\\
\left(\frac{\ln(k-1)R- \ln|x|}{\ln(k-1)}\right)^+   & |x|\ge R
\end{cases}
\]
with $k>2$ sufficiently large.

Finally, if $N\ge3$, let $\delta>0$ be small to be fixed, and
\[
t^*\varphi(x)=
\begin{cases}
1 & |x|\le R\\
((1+\delta)R^{N-2}|x|^{2-N}-\delta)^+   & |x|\ge R.
\end{cases}
\]
Letting $|\partial B_1|=\omega_N$ we obtain
\[
\begin{split}
\int_{\R^N}(|\D \varphi|^2-V(x)\varphi^2)dx &\le
\int_{\R^N\setminus B_R}|\D \varphi|^2\,dx - \int_{B_R}\eta\varphi^2\,dx \\
&\le (1+\delta)^2R^{2(N-2)}(N-2)^2 \omega_N\int_R^{+\infty} r^{2(1-N)}\cdot r^{N-1}\,dr - \frac{\omega_N}{N}
R^N\eta\\
&= \frac{R^{N-2}\omega_N}{N}\left( (1+\delta)^2 N(N-2) - R^2\eta\right) < 0
\end{split}
\]
by \eqref{eq:suff_cond_bottom}, if $\delta$ is sufficiently small.
\end{proof}
Now, let us prove that a local minimum solution exists.
By Proposition \ref{prop:exists_min}, we just need to show that, under the assumptions of
the theorem, $c_{V,\rho}<0$. Let $\varphi$ be as in \eqref{eq:V_neg_spect}. Notice that,
for every $t>0$,
\[
F(t\varphi) \le -\frac{t^p}{p}\|\varphi\|_p^p <0.
\]
Let $\bar t = \frac{R_*}{\|\nabla\varphi\|_2}$.
Then $\|\nabla \bar t \varphi\|_2 = R_*$ and $F(\bar t \varphi)<0$, hence
\eqref{eq:lem_rephrased} implies that
\[
\rho_* \le \|\bar t \varphi\|_2 = \frac{R_*}{\|\nabla\varphi\|_2}\rho,
\qquad\implies\qquad
\|\nabla\varphi\|_2 \le \frac{\rho}{\rho_*}R_* < R_*.
\]
Resuming, we have that both $\varphi \in S_\rho$ and
$\|\nabla\varphi\|_2 \le R_*$. By definition we infer
\[
c_{V,\rho} \le F(\varphi)<0
\]
and the theorem follows.
\qed

\subsection{Mountain pass solution}

The proof of the second part of Theorem \ref{Tmin}, i.e. the existence of a mountain pass
solution, can be obtained arguing as in the proof of Theorem \ref{T1}.
Some changes are in order, especially for the Palais-Smale condition, because in this framework
Lemma \ref{L>} cannot work. In this case, instead of working with a mountain pass geometry
uniform in $\rho$, as we did to find the minimizer, it is more convenient to use a mountain pass
geometry dependent on $\rho$; this will allow a direct comparison between the mountain pass level
and $m_\rho$, allowing to show that the Lagrange multipliers of the Palais-Smale sequence are
eventually positive.
  Observe that, by Remark \ref{rem:ass_SL}, \eqref{eq:ass_VW_con_rho} ensures that the assumptions of the Splitting Lemma  \ref{splitting-lemma} holds also for $r=+\infty$.

In the following, we assume that
$\sigma_i$, $\bar\sigma_i$,  $i=1,2$,
and $\tilde L$ in \eqref{eq:ass_VW_con_rho} are such that
\begin{align}
\label{eq:L1}
&\|V\|_{r} \cdot\rho^{(2-\frac{N}{r})\frac{2(p-2)}{N(p-2)-4}}\le  L_1,
&&\hspace{-5mm}
\|W\|_{s} \cdot\rho^{(1-\frac{N}{s})\frac{2(p-2)}{N(p-2)-4}}\le  L_1,
\\
\label{eq:L2}
&\|V\|_{r}\cdot\rho^{2 -\frac{N}{r}}\le  L_2,
&&\hspace{-5mm}
\|W\|_{s}\cdot\rho^{1-\frac{N}{s}} \le  L_2,
\\
\label{eq:L3}
&\|V\|_{r}\cdot\rho^{2 -\frac{N}{r}} \le L_3 m_\rho,
&&\hspace{-5mm}
\|W\|_{s}\cdot\rho^{1 -\frac{N}{s}} \le L_3 m_\rho,
\end{align}
for suitable positive constants $L_i$ to be chosen below independently of $V$,
$W$ and $\rho$ (this is possible in view of \eqref{eq:poho}). Moreover we notice that
\eqref{eq:L1} is equivalent to
\begin{equation}\label{eq:L1'}
\|V\|_{r}^{\frac{2r}{2r-N}} \le  L_1'\cdot\frac{m_\rho}{\rho^2},
\qquad
\|W\|_{s}^{\frac{2s}{s-N}} \le  L_1'\cdot\frac{m_\rho}{\rho^2}.
\end{equation}

\medskip

\emph{Mountain pass geometry - revisited}\medskip

We recall that by \eqref{1055} we have, for every $u\in S_\rho$,
\beq
\label{1250}
F(u)\ge \frac 12 \|\D u\|_2^2-\frac{G_q^2}{2}\|V\|_{\frac{q}{q-2}}\rho^{2-\frac{N(q-2)}{q}}\|\D u\|_2^{\frac{N(q-2)}{q}}
-\frac1p G^p\rho^{p-\gamma} \|\D u\|_2^\gamma,
\eeq
where, as usual, $q\in [2,2^*)$ satisfies $\frac{q}{q-2}=r$.
Let $\tilde R>0$ be such that
$$
\widetilde \cM_0:=\frac 12 \tilde R^2 -\frac1p G^p\rho^{p-\gamma} \tilde R^\gamma=\max_{t\ge 0}
\frac 12  t^2 -\frac1p G^p\rho^{p-\gamma} t^\gamma.
$$
Then $\tilde R = C \rho^{-\frac{p-\gamma}{\gamma-2}}$, where $C$ depends only on $N$ and $p$.
By a direct computation and \eqref{eq:poho} we have that $\widetilde \cM_0= 2 \widetilde M m_\rho$ for a suitable constant $\widetilde M>0$ independent on $\rho$.
Then we can choose $L_1>0$, only depending on $N$, $p$ and $q$, such that
\begin{equation}\label{eq:passaggio_L1}
\frac{G_q^2}{2}\|V\|_{\frac{q}{q-2}}\rho^{2-\frac{N(q-2)}{q}}\tilde R^{\frac{N(q-2)}{q}}
\le \widetilde M m_\rho
\qquad\iff\qquad
\|V\|_{r} \rho^{(2-\frac{N}{r})\frac{2(p-2)}{N(p-2)-4}}\le  L_1.
\end{equation}
By \eqref{1250}, \eqref{eq:passaggio_L1} and \eqref{eq:L1} we obtain
\beq
\label{1248}
 \widetilde \cM:=\frac 12 \tilde R^2 -\frac{G_q^2}{2}\|V\|_{\frac{q}{q-2}}\rho^{2-\frac{N(q-2)}{q}}\tilde R^{\frac{N(q-2)}{q}}
-  \frac1p G^p\rho^{p-\gamma} \tilde R^\gamma\ge \widetilde M m_\rho.
\eeq
Now, let us fix $u_0=(Z_\rho)_{h_0}, u_1=(Z_\rho)_{h_1}\in S_\rho$ such that
$$
\|\D u_0\|_2<\tilde R,\quad \|\D u_1\|_2>\tilde R,\quad F(u_0)<\widetilde \cM,\qquad F(u_1)<0,
$$
and define the mountain pass value
%sistemato un gamma-->xi
$$
  m_{V,\rho}:=\inf_{\xi\in\Gamma}\max_{t\in [0,1]}F(\xi(t)),\qquad
\Gamma=\{\xi:[0,1]\to S_\rho\ :\ \xi(0)=u_0,\
\xi(1)=u_1\}.
$$
By \eqref{1248} and \eqref{eq:base_ass} we infer
\beq
\label{1251}
\widetilde M m_\rho \le m_{V,\rho} < m_{\rho}
\eeq
(the strict inequality follows as in Remark \ref{mv-mr lemma}).
As in Step 1 of Proposition \ref{bounded PS}, we get a Palais-Smale sequence $(v_n)_n$, at the level $m_{V,\rho}$, that satisfies \eqref{Pohozaev-id}, \eqref{1731} and we have to verify that it is bounded,  the related Lagrange multipliers $(\lambda_n)_n$ are bounded and converge, up to a subsequence, to a positive value.\medskip

\emph{Bounded Palais-Smale sequence}\medskip

The proof of this step goes on as the proof of  Proposition \ref{bounded PS}  untill \eqref{a_n-bd}, with analogous notation.
In this framework, by Lemma \ref{stime_q V W} we have
\beq
\label{1733}
\begin{split}
[N(p-2)-4]a_n  \le &\,   2N(p-2)  m_{V,\rho}+N|4-p|c_n+4|d_n|+o(1)\\
 \le & \, 2N(p-2) m_{\rho}+N|4-p|\left(G_q^2\rho^{2 -\frac{N}{r}}\|V\|_{r}\right)a_n^{
 \frac{N}{2r}} \\
 &\, +4\left(G_{q_1}\|W\|_{s}\rho^{1-\frac{N}{s}}\right)a_n^{\frac12\left(1+\frac{N}{s}\right)}+o(1)
\end{split}
\eeq
where $q_1\in [2,2^*)$ satisfies $\frac{2 q_1}{q_1-2}=s$.
By assumption,
\beq
\label{1647}
\frac{N}{2r}<1\quad\mbox{ and }\quad \frac12\left(1+\frac{N}{s}\right)<1,
\eeq
hence when $a_n\ge 1$
\beq
\label{1734}
%\begin{split}
\left\{[N(p-2)-4]-N|4-p|  G_q^2\rho^{2 -\frac{N}{r}}\|V\|_{r} -%&
4  G_{q_1}\rho^{1-\frac{N}{s}}  \|W\|_{s}  \right\}a_n
%\\&
\le 2N(p-2)m_{\rho}+o(1).
%\end{split}
\eeq
Then we can choose $L_{ 2}$ in \eqref{eq:L2} small, in such a way that
\beq
\label{1644}
a_n\le \max\left\{1,\frac{3N(p-2)}{N(p-2)-4}m_\rho\right\},
\eeq
and in particular the sequence $a_n$ is bounded. We deduce that the
sequences $b_n,c_n, d_n$ and $\lambda_n$ are bounded as well, and they all converge,
up to subsequences, to suitable $a,b,c,d$ and $\lambda$, respectively.
Then we focus on the sign of the Lagrange multiplier $\lambda$.
\medskip

\emph{Lower bound for the Lagrange multiplier}\medskip

As in \eqref{lambda>0}, by  \eqref{1251} and the estimates of Lemma \ref{stime_q V W}
we obtain
\beq
\label{1640}
\begin{split}
\la\rho^2 &=  \frac{2N-(N-2)p}{N(p-2)-4}2m_{V,\rho}
-\frac{(N-2)(p-2)}{N(p-2)-4}c
-\frac{2(p-2)}{N(p-2)-4}d\\
&\ge
 C_1 \cdot 2\widetilde M m_\rho - C_2 \rho^{2 -\frac{N}{r}}\|V\|_{r}a^{\frac{N}{2r}} - C_3 \rho^{1-\frac{N}{s}}\|W\|_{s}a^{\frac12\left(1+\frac{N}{s}\right)},
\end{split}
\end{equation}
where the nonnegative constants $C_i$ only depend on $N,p,q,q_1$.
Now, if $a\ge1$, then  \eqref{1644} implies that $C m_\rho\ge 1$ too.
Then \eqref{1640}, \eqref{1647} and \eqref{1644} imply
\[
\begin{split}
\la\rho^2 &\ge \left[ C_1' - C_2' \rho^{2 -\frac{N}{r}}\|V\|_{r}
-C_3'\rho^{1-\frac{N}{s}}\|W\|_{s}\right]\, m_\rho,
\end{split}
\]
and using again \eqref{eq:L2}, with a possibly smaller value of $L_{ 2}$, we
infer that $\lambda>0$. If instead $a\le 1$, then we can use \eqref{eq:L3} to
write \eqref{1640} as
\[
\la\rho^2
\ge
 C_1 \cdot 2\widetilde M m_\rho - C_2 \rho^{2 -\frac{N}{r}}\|V\|_{r}
 - C_3 \rho^{1-\frac{N}{s}}\|W\|_{s}
 = \left[ C_1 \cdot 2\widetilde M - (C_2+C_3)L_3\right]\, m_\rho,
\]
and also in this case $\lambda>0$, provided $L_3$ in \eqref{eq:L3} is chosen sufficiently
small.
\medskip

\emph{Palais-Smale condition}\medskip

Up to now, we have shown that the Palais-Smale sequence $(v_n)_n$ provided by the mountain pass geometry and  the corresponding sequence of Lagrange multipliers $(\lambda_n)_n$ satisfy:
$$
F(v_n)\to m_{V,\rho},\qquad F'(v_n)\phi=\lambda_n\phi+o(1)\|\phi\|_{H^1},\qquad \lambda_n\to\lambda>0.
$$
Then $(v_n)_n$ is a Palais-Smale sequence also for $I_\lambda$ and by the Splitting Lemma we can write
\[
v_n(x)=v(x)+\sum_{i=1}^kw^i(x-y_n^i)+o(1)\quad\mbox{ in } H^1(\R^N)
\]
where each $w^i\in H^1(\R^N)$ satisfies $-\Delta w+\lambda w=|w|^{p-2}w$, for $i\in\{1,\ldots,k\}$, and  the weak limit $v$  satisfies $-\Delta w+(\lambda-V) w=|w|^{p-2}w$.
Arguing as in \eqref{splitt eq2}  we get
\beq
\label{1451}
m_{V,\rho}=F(v)+\sum_{i=1}^k F_\infty (w^i).
\eeq
Let us assume by contradiction
that $v_n\not\to v$ strongly in $H^1$, that is $k>0$. Then we denote
$$
\mu=\|v\|_2,\quad \alpha_i=\|w^i\|_2,
\qquad\text{ so that } \mu^2 + \sum_{i=1}^k \alpha_i^2 = \rho^2.
$$
Now, by Lemma \ref{lem:append} and \eqref{eq:poho},
\begin{equation}\label{eq:F(w)}
\sum_{i=1}^k F_\infty (w^i)\ge \sum_{i=1}^k m_{\|w_i\|_2} \ge m_{\alpha_1} =
m_\rho \left(\frac{\alpha_1}{\rho}\right)^{-2\theta},
\qquad
\theta:=\frac{2N-p(N-2)}{N(p-2)-4}>0.
\end{equation}
We claim that, taking into account \eqref{eq:L1'}, we have:
\begin{equation}\label{eq:claim_F(v)}
F(v) \ge - \theta m_\rho \left(\frac{\mu}{\rho}\right)^2.
\end{equation}
If this is the case, using  \eqref{eq:F(w)} and the fact that $\mu^2 + \alpha_1^2\le \rho^2$,
we obtain
\beq
\label{1558}
\begin{split}
F(v)+\sum_{i=1}^k F_\infty(w_i)&\ge  - \theta m_\rho \left(\frac{\mu
}{\rho}\right)^2 + m_\rho \left(\frac{\alpha_1}{\rho}\right)^{-2\theta}
\ge  m_\rho\left[ - \theta + \theta\left(\frac{\alpha_1}{\rho}\right)^2 +
\left(\frac{\alpha_1}{\rho}\right)^{-2\theta}\right]
\\
& \ge m_\rho \min_{0<x\le1}\left[ - \theta + \theta x +  x^{-\theta}\right] = m_\rho,
\end{split}
\eeq
which completes the proof of the compactness, because it is in contradiction
with \eqref{1451} and the fact that $m_{V,\rho}<m_\rho$ by \eqref{1251}.

To conclude, we are left to show \eqref{eq:claim_F(v)}. To this aim, let us set
$a=\|\D v\|_2^2$, $b=\|v\|_p^p$, $c=\int_{\R^N}V(x)v^2dx$, $d=\int_{\R^N}V(x)v\D v\cdot x\, dx$.
Taking into account the Pohozaev identity for $v$:
\[
a-\frac{N(p-2)}{2p}b-\frac{N}{2}c-d=0,
\]
we infer
\begin{equation}\label{eq:F(v)_ini}
\begin{split}
F(v)&=\frac12 a- \frac12 c-\frac{1}{p}b \ge
\frac{N (p - 2) - 4}{2 N (p-2)}a
-\frac{p - 4}{2  (p-2)}c
+\frac{2}{ N (p-2)}d\\
&
\ge C_0\left[2a - C_1 c - C_2 |d|\right],
\end{split}
\end{equation}
where the nonnegative constants $C_i$ only depend on $N$ and $p$. Now, using
Lemma \ref{stime_q V W}, we have that
\begin{equation}\label{eq:F(v)_1}
a - C_1 c \ge a -C_1'\|V\|_{r}\mu^{2-\frac{N}{r}}a^{\frac{N}{2r}}
\ge - C_1'' \|V\|_{r}^{\frac{2r}{2r-N}}\mu^2,
\end{equation}
where in the last step we used the elementary inequality
\[
a>0,\ 0<\tau<1,\ k>0
\qquad\implies\qquad
a-ka^\tau \ge -(1-\tau)\tau^{\frac{\tau}{1-\tau}}\, k^{\frac{1}{1-\tau}}.
\]
Analogously,
\begin{equation}\label{eq:F(v)_2}
a - C_2 |d| \ge a -C_2'\|W\|_{s}\mu^{1-\frac{N}{s}}a^{\frac12\left(1+\frac{N}{s}\right)}
\ge - C_2'' \|W\|_{s}^{\frac{2s}{s-N}}\mu^2.
\end{equation}
Substituting \eqref{eq:F(v)_1} and  \eqref{eq:F(v)_2} into  \eqref{eq:F(v)_ini},
and exploiting assumption \eqref{eq:L1'} we finally obtain
\begin{equation*}
F(v) \ge -C L'_1 m_\rho \frac{\mu^2}{\rho^2},
\end{equation*}
and \eqref{eq:claim_F(v)} follows by taking $L_1'$ sufficiently small.

\qed

%%%%%%%%%%%%%%%%%%%%%%%%    Acknowledgement   %%%%%%%%%%%%%%%%%%%%%%%%%%%%%%

{\small {\bf Acknowledgement}. The authors have been supported by  the INdAM-GNAMPA group.
R.M. acknowledges also the MIUR Excellence Department
Project awarded to the Department of Mathematics, University of Rome
Tor Vergata, CUP E83C18000100006. G.R. has been supported by
the Italian PRIN Research Project 2017 ``Qualitative and quantitative aspects of nonlinear PDE''.
G.V. is partially supported by the project Vain-Hopes within the program VALERE --
Universit\`a degli Studi della Campania ``Luigi Vanvitelli'' and by the Portuguese
government through FCT/Portugal under the project PTDC/MAT-PUR/1788/2020.
}

%%%%%%%%%%%%%%%%%%%%%%%%    Bibliografia    %%%%%%%%%%%%%%%%%%%%%%%%%%%%%%

%%%%%%%%%%%%%%%%%%%%%%%%%%%%%%%%%%%%%%%%%%%%%%%%%%%%%%%%%%%%%%%%%%%%
%%%%%%%%%%%%%%%%%%%%%%%%%%%%%%%%%%%%%%%%%%%%%%%%%%%%%%%%%%%%%%%%%%%%

\end{document}